\documentclass[12pt,letterpaper,reqno]{amsart}

\usepackage{times} 
\usepackage[T1]{fontenc} 
\usepackage{mathrsfs} 
\usepackage{latexsym} 
\usepackage[dvips]{graphics}
\usepackage{epsfig}
\usepackage{amsmath,amsfonts,amsthm,amssymb,amscd} 
\input amssym.def 
\input amssym.tex

\newcommand\be{\begin{equation}} 
\newcommand\ee{\end{equation}}
\newcommand\bea{\begin{eqnarray}} 
\newcommand\eea{\end{eqnarray}} 
\newcommand\bi{\begin{itemize}}
\newcommand\ei{\end{itemize}} 
\newcommand\ben{\begin{enumerate}} 
\newcommand\een{\end{enumerate}}
\newcommand\bc{\begin{center}} 
\newcommand\ec{\end{center}} 
\newcommand\ba{\begin{array}} 
\newcommand\ea{\end{array}}

\newcommand{\R}{\ensuremath{\mathbb{R}}} 

\newcommand{\Z}{\ensuremath{\mathbb{Z}}} 
\newcommand{\Q}{\mathbb{Q}} 
\newcommand{\N}{\mathbb{N}}

\newcommand{\bm}[1]{\mbox{\boldmath $ {#1} $}}

\newtheorem{thm}{Theorem}[section]

\newtheorem{lem}[thm]{Lemma}

\theoremstyle{definition} 

\begin{document}

\title[On $m$-covering families of Beatty sequences with irrational moduli] 
{On $m$-covering families of Beatty sequences with irrational moduli}


\author{Peter Hegarty} \address{Department of Mathematical Sciences, 
Chalmers University Of Technology and University of Gothenburg,
41296 Gothenburg, Sweden} \email{hegarty@chalmers.se}


\subjclass[2000]{00Z00, 00Z00 (primary), 00Z00 (secondary).} \keywords{Beatty 
sequence, Weyl criterion}

\date{\today}

\begin{abstract} We generalise Uspensky's theorem 
characterising eventual exact (e.e.) 
covers of the positive integers by homogeneous 
Beatty sequences, to e.e. 
$m$-covers, for any $m \in \N$, by homogeneous
sequences with irrational moduli. We also consider inhomogeneous sequences,
again with irrational moduli, 
and obtain a purely arithmetical characterisation of e.e. $m$-covers. This
generalises a result of Graham for $m = 1$, but when $m > 1$
the arithmetical description is more complicated. 
Finally we speculate on how one might make sense of the notion of an
exact $m$-cover when $m$ is not an integer, and 
present a $\lq$fractional version' of Beatty's theorem. 
\end{abstract}


\maketitle

\setcounter{equation}{0}

\setcounter{equation}{0}

\section{Introduction and statement of results}

Throughout this paper, the integer and fractional parts of a real number $x$ 
will be denoted by $\lfloor x \rfloor$ and $\{x\}$ respectively. Hence 
\be
\{x\} = x - \lfloor x \rfloor.
\ee
We trust that no confusion will arise from using the same notation for sets as 
for fractional parts of numbers.
\\
\par
Next, we define the terms in the title of the 
article.
\\
\\
{\bf Definition 1.1.}
Let $\alpha, \beta \in \R$ with $\alpha > 0$. Denote
\be
S(\alpha,\beta) := \{ \lfloor n\alpha + \beta \rfloor : n \in \N \}.
\ee
We wish to think of $S(\alpha,\beta)$ as a multiset of integers : 
in other words, if some 
integer appears more than once (which will be the case whenever $\alpha < 1$), 
then we take account of the number of times it appears. 
The multiset $S(\alpha,\beta)$ is called a {\em Beatty sequence}. The number 
$\alpha$ is called the {\em modulus} of the sequence. If $\beta = 0$ we say that
the Beatty sequence is {\em homogeneous}, otherwise {\em inhomogeneous}. 
Note that, if $\alpha \in \N$, then $S(\alpha,\beta)$ is an arithmetic
progression (AP). 
\\
\\
{\bf Definition 1.2.}
Let $m$ be a positive integer, $I$ a finite index set and $\{S_i : i \in I\}$ 
a family of multisets of integers. The 
family is said to be an {\em $m$-cover} 
if every integer appears at least $m$ times in the union of the $S_i$, counting
multiplicities. If every integer appears exactly $m$ times, we say that the 
$m$-cover is {\em exact}. A little more generally, if every sufficiently large 
positive integer appears at least (resp. exactly) $m$ times, we speak of an
{\em eventual (exact) $m$-cover}. Eventual exact $m$-covers are the primary 
objects of study in this paper, and we shall henceforth use the acronym 
$m$-EEC for these. In addition, we shall always drop the prefix when $m=1$. 
\\
\\
{\bf Remark 1.3.} 
It is not hard to see that an eventual (exact) $m$-covering family of
APs is in fact an (exact) $m$-cover. However, the same
need not be true of more general Beatty sequences.
\\
\\
{\bf Definition 1.4.} Let $m > 1$ and $\{S_i : i \in I\}$ be an $m$-EEC. 
We say that this covering family is {\em reducible} if there exist
positive integers $m_1,m_2$ satisfying $m_1 + m_2 = m$ 
and a partition $I = J \sqcup K$, such that
$\{S_i : i \in J\}$ is an $m_1$-EEC and $\{S_i : i \in K\}$ is an
$m_2$-EEC. Otherwise, the cover is called {\em irreducible}. 
\\
\par  
The basic problem of interest is to characterise all $m$-EEC's 
consisting of Beatty sequences. The main new results of this paper provide
such characterisations for all $m \in \N$, when the moduli of
the sequences are all irrational. 
\par We begin with a brief survey of the
existing literature. Henceforth, it is to be understood that 
$\lq$cover' always refers to a covering family of Beatty sequences. 
It is clear that a necessary condition for
the family $\{S(\alpha_i,\beta_i) : i = 1,...,k\}$ to be 
an $m$-EEC is that
\be\label{density}
\sum_{i=1}^{k} \frac{1}{\alpha_i} = m.
\ee
There is a considerable literature on the case
$m=1$ - 
for a recent overview and a much more exhaustive list of references than
those given here, see 
Section 10 of \cite{F09}. 
In the case of homogeneous sequences, there is a classical
result :
\\
\\
{\bf Theorem 1.5.} 
{\em Let $\alpha_1,...,\alpha_k$ be positive real numbers. Then 
$\{S(\alpha_1,0),...,S(\alpha_k,0)\}$ is an 
EEC if and only if (\ref{density}) holds and either
\par (i) $k=1$ and $\alpha_1 = 1$, or
\par (ii) $k=2$ and $\alpha_1 \not\in \Q$.}
\\
\par The sufficiency of condition (i) is trivial, that of (ii) is known as
Beatty's theorem, though it was first discovered by Lord 
Rayleigh{\footnote{Condition (ii) guarantees that {\em every} positive integer 
occurs exactly once in the multiset 
$S(\alpha_1,0) \cup S(\alpha_2,0)$, and Beatty's 
theorem is usually stated in this form.}}. That 
$k \leq 2$ is necessary was first proven by Uspensky \cite{U}, using
Kronecker's approximation theorem. A more elementary proof
was later provided by Graham \cite{Gr63}. 

When one allows inhomogeneous sequences, there is no such simple 
classification. However, a certain amount is known. In the case of two
sequences with irrational moduli, there is the following generalisation of
Beatty's theorem :
\\
\\
{\bf Theorem 1.6. (Skolem \cite{S}, Fraenkel \cite{F69})}
{\em Let $\alpha_1,\beta_1,\alpha_2,\beta_2$ be real numbers, with 
$\alpha_1, \alpha_2$ positive, irrational and satisfying (\ref{density}). Then 
$\{S(\alpha_1,\beta_1), S(\alpha_2,\beta_2)\}$ is an 
EEC if and only if 
\be\label{inhompair}
\frac{\beta_1}{\alpha_1} + \frac{\beta_2}{\alpha_2} \in \Z.
\ee}
 
Let $\{S(\alpha_1,\beta_1), S(\alpha_2,\beta_2)\}$ be an EEC and suppose
$\{S(a_i,\phi_i) : i = 1,...,\mu\}$ and $\{S(c_{j},\psi_{j}) : 
j = 1,...,\nu\}$ are exact covering families 
of arithmetic progressions. Then, clearly, 
\be\label{graham}
\left\{ \bigcup_{i=1}^{\mu} S(\alpha_1 a_i, \alpha_1 \phi_i + \beta_1) \right\} 
\cup \left\{ \bigcup_{j=1}^{\nu} S(\alpha_2 c_{j}, 
\alpha_2 \psi_{j} + \beta_2) \right\}
\ee
is also an EEC. Graham \cite{Gr73} proved that any EEC in which at least one
of the moduli is irrational must have the form (\ref{graham}). In 
particular, this implies that the moduli in an EEC are either all rational or 
all irrational. It also reduces the classification of EEC's with 
irrational moduli to that of EEC's with integer moduli, that is, 
of exact covering families of APs. 
The latter problem has a long history but remains inadequately resolved. 
For an introduction to known results and open problems
concerning covers and exact covers by APs, see Problems 
F13-14 in \cite{Gu}.   
One noteworthy fact is that the moduli in a covering family of
APs cannot all
be distinct. A beautiful proof of this, using generating functions, can be 
found in \cite{E}. Graham's 1973 result implies that the same is true
of EEC's with irrational moduli. An important open problem
in this field concerns EEC's with distinct rational moduli. 
Fraenkel \cite{F73} conjectured the \\
following :
\\
\\
{\bf Fraenkel's Tiling Conjecture.} {\em Let $0 < \alpha_1 < \cdots < 
\alpha_k$ and let $\beta_1,...,\beta_k$ be any real numbers. Then the 
family $\{S(\alpha_i,\beta_i) : i = 1,...,k\}$ is an EEC if and only
if $k \geq 3$ and 
\be
\alpha_i = \frac{2^k - 1}{2^{k-i}}, \;\;\; i = 1,...,k.
\ee}

So let us turn to $m > 1$. Now one is interested in characterising 
irreducible $m$-EECs. In the case of APs, the 
existence of irreducible exact $m$-covers, for every $m > 1$, was first 
demonstrated by Zhang Ming-Zhi \cite{Z}. Graham and O'Bryant
\cite{GrOB} studied $m$-EEC's with rational moduli, and proposed a 
generalisation of Fraenkel's Tiling Conjecture. The remainder of this paper
is concerned with irrational moduli. The only result we could find in the 
literature is the following generalisation of Beatty's theorem :
\\
\\
{\bf Theorem 1.7.}
{\em Let $m \in \N$ and $\alpha_1, \alpha_2$ be positive irrational numbers 
satisfying $1/\alpha_1 + 1/\alpha_2 = m$. Then every positive integer
appears exactly $m$ times in the multiset union $S(\alpha_1,0) \cup
S(\alpha_2,0)$.}
\\
\par This result seems to first appear in \cite{OB}. The proof given there is
not difficult, but employs generating functions. A completely 
elementary proof was given by Larsson \cite{L}, whose motivation
for studying $m$-covers came from combinatorial games. Note that
Theorem 1.7 implies that irreducible, homogeneous 
$m$-EEC's with irrational moduli do exist
for every $m > 1$. It turns out, however, that Theorem 1.7 describes all
of them. The first main result of this paper is the following :
\\
\\
{\bf Theorem 1.8.}
{\em Let $m \in \N$. Let $\alpha_1,...,\alpha_k$ be positive irrational numbers 
satisfying (\ref{density}). Then $\{S(\alpha_1,0),...,S(\alpha_k,0)\}$ is an 
$m$-EEC if and only if $k$ is even, $k = 2l$ say, and the $\alpha_i$ can be 
re-ordered so that
\be
\frac{1}{\alpha_{2i-1}} + \frac{1}{\alpha_{2i}} \in \Z, \;\;\; i = 1,...,l.
\ee}

From this we shall deduce
the following generalisation of Theorem 1.5 :
\\
\\
{\bf Theorem 1.9.} 
{\em 
Let $m \in \N$ and $\alpha_1,...,\alpha_k$ be positive real numbers, not all 
rational. Then $\{S(\alpha_1,0),...,S(\alpha_k,0)\}$ is an irreducible 
$m$-EEC if and only if $k=2$ and (\ref{density}) holds.}
\\
\par Turning to the inhomogeneous case, Theorem 1.6 generalises verbatim to 
$m > 1$. Since we could not find this fact stated explicitly anywhere in the 
literature, and our proof of it follows a different approach from that
in \cite{F69}, we state it as a separate result :
\\
\\
{\bf Theorem 1.10.}
{\em Let $m \in \N$. Let $\alpha_1, \beta_1, \alpha_2, \beta_2$ be real 
numbers, 
with $\alpha_1, \alpha_2$ positive, irrational and satisfying (\ref{density}). 
Then $\{S(\alpha_1,\beta_1), S(\alpha_2,\beta_2)\}$ is an $m$-EEC if and only
if (\ref{inhompair}) holds.}
\\
\par
Given the previous three results, it is natural to ask if  
Graham's 1973 result generalises as follows :
\\
\\
{\bf Question 1.11.} {\em Let $m \in \N$. Is it true that every $m$-EEC
with irrational moduli has the form
\be\label{grahamstruct}
\bigcup_{k=1}^{t} \left\{ \left\{ \bigcup_{i=1}^{\mu_k} S(\alpha_{2k-1}a_{i,k}, 
\alpha_{2k-1}\phi_{i,k} + \beta_{2k-1}) \right\} \cup 
\left\{ \bigcup_{j=1}^{\nu_{k}} S(\alpha_{2k}c_{j,k},
\alpha_{2k}\psi_{j,k} + \beta_{2k}) \right\} \right\}
\ee
where there are positive integers $m_1,...,m_t, d_1,...,d_t$ satisfying
$m_1d_1 + \cdots + m_td_t = m$, and, for each $k = 1,...,t$, one has 
\par (i) \be
\frac{1}{\alpha_{2k-1}} + \frac{1}{\alpha_{2k}} = m_k,
\ee
\par (ii) 
$\{S(a_{i,k},\phi_{i,k}) : i = 1,...,\mu_k\}$ and
$\{S(c_{j,k},\psi_{j,k}) : j=1,...,\nu_{k}\}$ are 
exact $d_k$-covering families of APs
\par (iii) \be
\frac{\beta_{2k-1}}{\alpha_{2k-1}} + 
\frac{\beta_{2k}}{\alpha_{2k}} \in \Z
\; ?
\ee}

The second main result of our paper is a negative answer to this question. 
We shall give explicit counterexamples and provide a description, in terms of 
APs, of the most 
general possible form 
of an $m$-EEC with irrational moduli (see Section 4 below). 
While this provides a $\lq$purely 
arithmetical/combinatorial' 
characterisation of such $m$-EECs, generalising \cite{Gr73}, we 
do not find our result satisfactory and feel that a simpler and more explicit 
description may be possible. This point will be discussed again later on.  
\par The rest of the paper is organised as follows. In Section 2, we give 
prerequisite notation, terminology and background results. As well as 
extending known theorems, a secondary purpose of our paper is to provide a 
uniform treatment of this material, something which we have found 
lacking in the existing literature. Our approach is based on Weyl's 
equidistribution theorem, and is thus most similar in spirit to that
followed by Uspensky \cite{U}. However, he only 
employed a weaker equidistribution result (Kronecker's theorem), and we 
also make more explicit the formula for the representation function $r(N)$, 
which counts the number of occurrences of the integer $N$ in a covering
family, in terms of 
sums of fractional parts (see eq. (\ref{rep}) below). 
Already in Section 2, we will prove Theorem 1.10 - 
this proof is extremely simple and provides the reader with a quick glimpse of 
our method. Section 3 deals with homogeneous Beatty sequences and the proof of 
Theorems 1.8 and 1.9. This section is the heart of the paper.  
In Section 4, we turn to the inhomogeneous case and the issue of 
how to properly generalise \cite{Gr73}. In Section 5, we briefly 
broach the question of
how one might make sense of the notion of $m$-cover, when $m$ is not an 
integer. What we
will actually prove is a fractional version of Beatty's theorem. This 
follows a suggestion of Fraenkel, who was also interested in possible
connections to combinatorial games. Further development of this line of 
investigation is left for future work, the possibilities for which we 
summarise in Section 6. 
  
\setcounter{equation}{0}

\section{Preliminaries and proof of Theorem 1.10}

$\;$ \par
Our approach is based on  
standard results concerning equidistribution of sequences. We have
chosen the following formulation as it seems the most natural one, even if we 
could get away with something less (see Remark 2.2 below) : 

\begin{lem}
Let $\tau_1 = \frac{p_1}{q_1},...,\tau_k = \frac{p_k}{q_k}$ be  
rational numbers written in lowest terms, whose denominators 
are co-prime, i.e.: GCD$(p_i,q_i)$ \\ $= 1$ for $i = 1,..,k$ and 
GCD$(q_i,q_j) = 1$ for $i \neq j$. Let $\theta_1,...,\theta_l$ be 
irrational numbers which are affine independent over $\mathbb{Q}$, 
i.e.: the equation
\be
c_0 + c_1 \theta_1 + \cdots + c_l \theta_l = 0, \;\;\;\; c_0,c_1,...,c_l \in 
\mathbb{Q},
\ee
has only the trivial solution $c_0 = c_1 = \cdots = c_l = 0$. 
\par For each $i = 1,...,k$, let $\mu_i$ be the measure on $[0,1)$ which 
gives measure $1/q_i$ to each point mass $u/q_i$, $u = 0,1,...,q_i - 1$. 
Let $\mu_0$ be Lebesgue measure and let $\mu$ be the measure on 
$[0,1)^{k+l}$ given by the product 
\be
\mu = \mu_1 \times \cdots \times \mu_k \times \mu_{0}^{l}. 
\ee
Then, as $n$ ranges over the natural numbers, the $(k+l)$-tuple
\be
\left( \{n \tau_1 \},...,\{n \tau_k\},\{n \theta_1\},...,\{n \theta_l\} \right)
\ee
is equidstributed on $[0,1)^{k+l}$ with respect to $\mu$. 
\end{lem}

\begin{proof}
When all the moduli are rational $(l = 0)$, this is just the Chinese 
Remainder Theorem. For general $l > 0$, the lemma thus 
asserts that the $l$-tuple
\be\label{ltuple}
\left( \{n \theta_1\},...,\{n \theta_l\} \right)
\ee
is equidistributed on $[0,1)^{l}$, when $n$ runs through any infinite
arithmetic progression. This fact can be immediately deduced from the
multi-dimensional Weyl criterion - see \cite{KN}, for example. 
\end{proof}
$\;$ \\
{\bf Remark 2.2.} As previously noted, we will not be needing the full force of
the lemma. What we will use is the consequence that, for any subintervals
$I_1,...,I_l$ of $[0,1)$ and any arithmetic progression $S(a,b)$, there
are arbitrarily large $n \in S(a,b)$ for which the $l$-tuple (\ref{ltuple}) 
lies in $I_1 \times \cdots \times I_l$.
\\
\par Fix $m,k \in \N$. Let real numbers 
$\alpha_i, \beta_i$, $i = 1,...,k$, be given with the 
$\alpha_i$ positive, irrational and satisfying (\ref{density}). 
To simplify notation, 
put 
\be\label{deftheta}
\theta_i := \frac{1}{\alpha_i}, \;\;\; \gamma_i := - \frac{\beta_i}{\alpha_i}, 
\;\;\; i = 1,...,k.
\ee
Hence, 
\be\label{thetadensity}
\sum_{i=1}^{k} \theta_i = m.
\ee
For $N \in \N$ and $i \in \{1,...,k\}$, set
\be
r_{i}(N) := \# \{n \in \N : \lfloor n\alpha_i + \beta \rfloor = N \}.
\ee
Setting
\be\label{repfn}
r(N) := \sum_{i=1}^{k} r_{i} (N),
\ee
we note that the family $\{S(\alpha_i,\beta_i) : i = 1,...,k\}$ is an $m$-EEC if
and only if $r(N) = m$ for all $N \gg 0$. The function $r(\cdot)$ will be 
called the {\em representation function} associated to the family 
$\{S(\alpha_i,\beta_i) : i = 1,...,k\}$. 
\par For each $i$, since $\alpha_i$ is irrational, there is at most 
one integer $n_i$ such that $n_i \alpha_i + \beta_i \in \Z$. Hence
$n\alpha_i + \beta_i \not\in \Z$ for all $n \gg 0$ and all $i$. It follows
that, for $N \gg 0$, 
\be
r_{i}(N) = \# \{n \in \N : N < n\alpha_i + \beta_i < N+1 \},
\ee
the point being that both inequalities are strict. One easily deduces that
\be\label{repdiff}
r_{i}(N) = \lfloor (N+1)\theta_i + \gamma_i \rfloor - \lfloor N\theta_i + 
\gamma_i \rfloor.
\ee
Define the function $\epsilon : \Z \rightarrow \R$ by 
\be\label{epsilon}
\epsilon (N) := \sum_{i=1}^{k} \{ N\theta_i + \gamma_i \}.
\ee
From (1.1), (\ref{thetadensity}) and (\ref{repdiff}) one easily deduces that
\be\label{rep}
r(N) = m + (\epsilon(N) - \epsilon(N+1)).
\ee
Hence, the Beatty sequences form an $m$-EEC if and only if the 
function $\epsilon(N)$ is constant for all $N \gg 0$. 
We can already quickly deduce Theorem 1.10. For if we
have only two sequences, then since $\theta_1 + \theta_2 \in \Z$ one has
\be
\epsilon (N) = \left\{ \begin{array}{lr} \{\gamma_1 + \gamma_2\}, & 
{\hbox{if $\{N\theta_1 + \gamma_1\} < \{\gamma_1 + \gamma_2\}$}}, \\
1 + \{\gamma_1 + \gamma_2\}, & {\hbox{otherwise}}. \end{array} \right.
\ee
It follows that, if $\gamma_1 + \gamma_2 \in \Z$, then 
$\epsilon(N) = 1$ for all $N \in \Z$, whereas if $\gamma_1 + \gamma_2 \not\in 
\Z$ then, since $\theta_1 \not\in \Q$,
a very weak form of Lemma 2.1 (already known to Dirichlet) implies
that $\{N\theta_1 + \gamma_1\} - \{\gamma_1 + \gamma_2\}$ will be both
positive and negative for arbitrarily large $N$.

\setcounter{equation}{0}

\section{The homogeneous case - proofs of Theorems 1.8 and 1.9.}

\par The proof of Theorem 1.8 will exhibit the main ideas of this paper, so
we will present it in detail, which will allow us to be more brief 
with all subsequent proofs. So let's now assume that all our 
sequences are homogeneous. Hence $\beta_i = \gamma_i = 0$ for $i = 1,...,k$ and 
\be
\epsilon(N) = \sum_{i=1}^{k} \{ N \theta_i\}.
\ee

Let $V$ be the vector space over $\Q$ spanned by $1,\theta_1,...,\theta_k$. 
Since the $\theta_i$ are irrational, we know
that dim$(V) > 1$. Let dim$(V) := d+1$ and, without loss of generality, assume 
that $1,\theta_1,...,\theta_d$ form a 
basis for $V$. Hence there exist rational numbers $q_{j,i}$, 
$0 \leq j \leq d$, $1 \leq i \leq k$ such that 
\be\label{q}
\theta_i = q_{0,i} + \sum_{j=1}^{d} q_{j,i} \theta_j, \;\;\; i = 1,...,k,
\ee
where
\be
1 \leq i \leq d \; \Rightarrow \; q_{j,i} = \left\{ \begin{array}{lr} 1, & 
{\hbox{if $j=i$}}, \\ 0, & {\hbox{if $j \neq i$}} \end{array} \right.
\ee
and
\be\label{sumq}  
\sum_{i=0}^{k} q_{j,i} = \left\{ \begin{array}{lr} m, & {\hbox{if $j = 0$}}, 
\\ 0, & {\hbox{if $j > 0$}}. \end{array} \right.
\ee
Set 
\be
Q_{j,i} := \left\{ \begin{array}{lr} \{ q_{j,i} \}, & {\hbox{if $j = 0$}}, 
\\ q_{j,i}, & {\hbox{if $j > 0$}}. \end{array} \right.
\ee
We may write each of the numbers $Q_{j,i}$ as a fraction in lowest terms, say 
\be
Q_{j,i} = \frac{u_{j,i}}{v_{j,i}}, \;\;\; u_{j,i} \in \Z, \;
v_{j,i} \in \N, \; {\hbox{GCD}}(u_{j,i},v_{j,i}) = 1.
\ee

We shall prove Theorem 1.8 by induction on $m$. The case $m=1$ follows from 
Theorem 1.5. If $m > 1$ then, in order to apply the induction hypothesis, it 
suffices, by Theorem 1.7, to find 
any pair $i_1,i_2 \in \{1,...,k\}$ such that 
$\theta_{i_1} + \theta_{i_2} \in \Z$. Hence this is all we need to do to finish 
the proof. Using Lemma 2.1, we shall deduce it as a consequence of the 
requirement that the function $\epsilon(N)$, given by (\ref{epsilon}), be 
constant for 
all $N \gg 0$. In a way which we will make rigorous 
in what follows, that lemma will allow us to ignore the influence of 
all but one of 
$\theta_1,...,\theta_d$ - for simplicity, we select $\theta_1$ (see eqs.
(\ref{equi1}) and (\ref{equi2})) - and then 
reduce the proof of the theorem to a purely
combinatorial problem (Proposition 3.3 below). 
\par To begin with, define positive integers $L_0, L$ by 
\be\label{l}
L_0 := {\hbox{LCM}} \{v_{0,i} : q_{1,i} \neq 0 \}, \;\;\;
L := {\hbox{LCM}} \{ |u_{1,i}| : q_{1,i} \neq 0\}.
\ee
For each $i$ such that $q_{1,i} \neq 0$, define the numbers 
$U_{i}, V_{i}$ by 
\be\label{uv}
Q_{0,i} =: \frac{U_{i}}{L_0}, \;\;\; q_{1,i} =: \frac{L}{V_i}.
\ee
Finally, we set
\be\label{homsys}
\left\{ \begin{array}{lr} 
a_i := L_0 V_i, \;\; b_i := -U_{i}V_i, & {\hbox{if $q_{1,i} > 0$}}, \\ 
c_{j} := - L_0 V_j, \;\; d_{j} := - U_{j} V_j, & {\hbox{if $q_{1,j} < 0$}}.
\end{array} \right.
\ee

We shall use Lemma 2.1 to establish the following claim :
\\
\\
{\bf Claim 3.2.} {\em If the function $\epsilon(N)$ is constant for all 
$N \gg 0$, then for every $t \in \Z$, we have an equality of multisets
\be\label{apunion}
\bigcup_{q_{1,i} > 0} S(a_i,t b_i) = \bigcup_{q_{1,j} < 0} 
S(c_{j},t d_{j}).
\ee}

Suppose the claim were false. Then clearly it must fail for some non-negative 
$t$. Choose such a $t$ 
and let $\eta_t$ be an element of the multiset difference. 
Without 
loss of generality, $\eta_t$ occurs more often on the left-hand side of 
(\ref{apunion}), 
say $r$ times on the left-hand side and $s$ times on the right-hand side, with 
$r > s$. Now let $\delta$ be a sufficiently small, positive real number - how 
small it should be will become clear below. By Lemma 2.1, we can find 
arbitrarily large integers $n$ satisfying
\be\label{equi1}
n \equiv 1 \; ({\hbox{mod $L_0 L$}}), \;\;\; \delta < \{ n \theta_1\} < 
\delta + e^{-1/\delta}, \;\;\; \{ n \theta_i\} < \delta^3, \; i = 2,...,d.
\ee
Let $n_0$ be any positive integer satisfying (\ref{equi1}).
Let $N_+$ (resp. $N_-$) be the least positive integer which is divisible by 
$n_0$, congruent to $t$ modulo $L_0 L$ and greater than 
$\frac{1}{\delta L_0 L} n_0 \eta_t$ (resp. $\frac{1-\delta}{\delta L_0 L} n_0 
\eta_t$). Then the point is that, provided $\delta$ is sufficiently 
small, for every $i = 1,...,k$ we have  
\be
\lfloor N_+ \{\theta_i \} + \{N_+ q_{0,i} \} \rfloor -
\lfloor N_- \{\theta_i \} + \{N_- q_{0,i} \} \rfloor = 
\left\{ \begin{array}{lr} 1, & 
{\hbox{if $q_{1,i} > 0$ and $\eta_t \in S(a_i,t b_i)$}}, \\
-1, & {\hbox{if $q_{1,i} < 0$ and $\eta_t \in S(c_{i}, t d_{i})$}}, 
\\ 0, & {\hbox{otherwise}}. \end{array} \right.
\ee
This in turn is easily seen to imply that 
\be
\epsilon(N_+) - \epsilon(N_-) = s - r \neq 0.
\ee
Since the numbers $N_+$ and $N_-$ can be made arbitrarily large, this would 
mean that the function $\epsilon(N)$ could not be constant for $N \gg 0$, a 
contradiction which establishes Claim 3.2. 
  
We state the next assertion as a separate proposition, as the reader may find 
it interesting in its own right. It is also the crucial combinatorial 
ingredient in this section :
\\
\\
{\bf Proposition 3.3.} {\em Let $a_1,...,a_{\mu}, c_{1},...,c_{\nu}$ be positive 
integers and $b_1,...,b_{\mu}, d_{1},...,d_{\nu}$ be any integers. If, for every 
$t \in \Z$, we have an equality of multisets 
\be\label{apunion2}
\bigcup_{i=1}^{\mu} S(a_i,t b_i) = \bigcup_{j=1}^{\nu} S(c_{j}, t d_{j}),
\ee
then $\mu = \nu$, and we can reorder so that, for each $i = 1,...,\mu$, 
$a_i = c_{i}$ and $b_{i} \equiv d_{i} \; ({\hbox{mod $a_i$}})$. }
\\
\par The proof of the proposition will employ the following facts :
\\
\\
{\bf Lemma 3.4.} {\em Let $p$ be a prime, $l$ a non-negative integer, 
$l_1,...,l_{\chi}$ integers each strictly greater than $l$ and 
$b,d_1,...,d_{\chi}$ any integers. Suppose that, as sets,
\be\label{union}
S(p^{l},b) \subseteq \bigcup_{j=1}^{\chi} S(p^{l_j},d_j).
\ee
Then, 
\par (i) $S(p^l,b)$ equals the disjoint union of 
some subset of the terms on the right-hand side of (\ref{union}),
\par (ii) Let $L$ be the maximum of the $l_j$. Then for some $\xi_1 \in 
\{0,1,...,p^L - 1\}$ there exists, for each $\xi_2 \in \{0,1,...,p-1\}$, some 
$j$ such that 
\be
l_j = L \;\;\;\; {\hbox{and}} \;\;\;\; d_j \equiv \xi_1 + \xi_2 p^{L-1} \; 
({\hbox{mod $p^L$}}).
\ee}

\begin{proof} {\em of Lemma 3.4.}
These are standard observations which can be proven in various ways. 
For example, one can consider the $p$-ary rooted tree $\mathcal{T}$, 
whose nodes 
are all the progressions $S(p^i,u)$, where $0 \leq i \leq L$ and $0 \leq u
< p^i$, and in which, for $i < L$, the node $S(p^i,u)$ has the $p$ daughters
$S(p^{i+1},u + vp^i)$, $v = 0,1,...p-1$. Eq. (\ref{union}) expresses the 
hypothesis that the rooted subtree $T_0$ under a certain node $x$ is, 
apart from the node $x$ itself, entirely contained inside
the union of a collection $T_1,...,T_{\chi}$ of rooted subtrees at strictly 
lower levels. Part (i) then 
asserts that some subset of the 
$T_1,...,T_{\chi}$ are pairwise disjoint and their
union equals $T_{0} \backslash \{x\}$. This is simple to prove, for example by
induction on the depth of $T_0$. Part (ii) is then also an immediate
consequence of the rooted tree structure. 
\end{proof}

\begin{proof} {\em of Proposition 3.3.}
We shall perform an induction on several different parameters. First of all, 
let $n$ be the total number of distinct primes which divide at least one of 
the moduli $a_i$ or $c_{j}$. If $n = 0$ then 
each individual AP is just $\Z$ and the proposition simply 
asserts the obvious fact that they must then be equal in number, i.e.: that 
$\mu = \nu$. So now suppose $n > 0$ and that the proposition is true for all 
smaller values of $n$. Let $p := p_1 < \cdots < p_n$ be the distinct primes 
which divide at least one modulus. Let $p^k$ denote the highest power of $p$ 
dividing any modulus and partition the moduli into subsets 
$M_0, M_{0}^{\prime},...,M_k,M_{k}^{\prime}$, where
\be
M_{l} := \{ i : p^{l} \mid \mid a_i \}, \; M_{l}^{\prime} := 
\{j : p^{l} \mid \mid c_{j}\}, \;\;\; l = 0,...,k.
\ee

By the Chinese Remainder Theorem, for each $i = 1,...,\mu$ (resp. each 
$j = 1,...,\nu$) we can write 
\be
S(a_i,t b_i) = S(p^{l_i},t b_i) \cap S(A_i,t b_i) \;\;\; 
({\hbox{resp.}} \; S(c_{j},t d_{j}) = S(p^{l_{j}^{\prime}},t d_{j}) \cap 
S(C_{j},t d_{j})),
\ee 
where $p^{l_i} \mid \mid a_i$ and $A_i = a_i/p^{l_i}$ 
(resp. $p^{l_{j}^{\prime}} \mid \mid c_{j}$ and $C_j = c_j/p^{l_{j}^{\prime}}$). 
Let $\xi \in \{0,1,...,p^k - 1\}$ and let $t$ be any integer s.t. 
$t \equiv 1 \; ({\hbox{mod $p^k$}})$. Considering the intersection of both 
sides of (\ref{apunion2}) with 
$S(p^k,\xi)$ we find that, as multisets,  
\be\label{oneprimeless}
\bigcup_{i : b_i \equiv \xi \; ({\hbox{mod $p^k$}})}
S(A_i,t b_i) = \bigcup_{j : d_j \equiv \xi ({\hbox{mod $p^k$}})} S(C_j,t d_j).
\ee

Now note that a necessary and sufficient condition for (\ref{apunion2}) 
to hold for every 
$t \in \Z$ is that it do so for any $t$ divisible only by those primes dividing 
some $a_i$ or $c_j$. Applying this observation to (\ref{oneprimeless}) 
instead, we deduce that the latter equality 
holds for every $t \in \Z$. Since there are 
exactly $n-1$ primes 
dividing some $A_i$ or $C_j$, we can apply the induction hypothesis to 
conclude that, for each $i$ such that $b_i \equiv \xi \; ({\hbox{mod $p^k$}})$,
there exists a $j$ such that $S(A_i,b_i) = S(C_j,d_j)$. For such a pair 
$(i,j)$ it follows that
\be\label{containment}
S(a_i,b_i) \supseteq S(c_j,d_j) \Leftrightarrow l_i \leq l_{j}^{\prime}.
\ee

Now we introduce the second induction parameter, which is the total number of 
APs involved in (\ref{apunion2}), i.e.: on the quantity 
$\mu + \nu$. 
It is clear that Proposition 3.3 holds if $\mu = \nu = 1$, so suppose 
$\mu + \nu > 2$ and that the proposition holds for any smaller value of 
$\mu + \nu$. If there were any pair $(i,j)$ whatsoever such that 
$S(a_i,b_i) = S(c_j,d_j)$, then we could immediately cancel this pair from 
(\ref{apunion2}) and apply the induction on $\mu + \nu$ to deduce the 
proposition. Hence, 
we may assume no such pair exists.
\par Let $l_{\min}$ (resp. $l_{\min}^{\prime}$) denote the smallest value of $l$ 
(resp. $l^{\prime}$) such that the set $M_l$ (resp. $M_{l}^{\prime}$) is 
non-empty. We claim that $l_{\min} = l_{\min}^{\prime}$. To see this, set 
$t := p^k$ in (\ref{apunion2}) and consider the contribution of both sides to numbers 
which are divisible by $p^{l}$ but not $p^{l+1}$, where $l = \min \{l_{\min}, 
l_{\min}^{\prime}\}$. These contributions cannot be equal if 
$l_{\min} \neq l_{\min}^{\prime}$, since then only one side would give a 
non-empty contribution. In fact, we can deduce much more. Let $l := l_{\min}$. 
It is clear that, for every $t^{*} \in \Z$, we have equality of multisets
\be\label{minl}
\bigcup_{p^{l} \mid \mid a_i} S(A_i,t^{*} b_i) = \bigcup_{p^{l} \mid \mid c_j} 
S(C_j,t^{*} d_j).
\ee 
By induction on the first parameter $n$, the total number of prime divisors of 
the $a_i$ and $c_j$, we can deduce that the progressions 
$S(A_i,b_i)$ for which $p^{l} \mid \mid a_i$ and the 
progressions $S(C_j,d_j)$ for which $p^{l} \mid \mid c_j$ are 
equal in pairs. This fact will be exploited later on.  
\par For the next step in the argument, consider any $i$ for which 
$l_i = l_{\min}$. For each $\xi$ such that $S(p^{l_i},b_i) \supseteq S(p^k,\xi)$ 
we can find, as shown earlier, some $j$ such that $S(A_i,b_i) = S(C_j,d_j)$ and 
$S(a_i,b_i) \supseteq S(c_j,d_j)$. Clearly, the multiset union of all these 
$S(c_j,d_j)$ must contain $S(a_i,b_i)$ and thus (\ref{containment}) and 
Lemma 3.4(i) imply 
that some subset of the $S(c_j,d_j)$ are pairwise disjoint and their 
union equals $S(a_i,b_i)$. To summarise, for any $i$ such that $l_i = l_{\min}$,
we can find a set of $j$'s such that
\be\label{intermediate}
S(C_j,d_j) = S(A_i,b_i) \; {\hbox{for each $j$ and}} \; S(a_i,b_i) = 
\bigsqcup_{j} S(c_j,d_j).
\ee
These conditions imply that
\be\label{modulop}
S(p^{l_i},b_i) = \bigsqcup_{j} S(p^{l_{j}^{\prime}},d_j).
\ee
\par If, in (\ref{intermediate}), 
we had $S(a_i,b_i) = S(c_j,d_j)$ for some $j$, then we could
apply the induction on $\mu + \nu$. Hence we may assume that $l_{j}^{\prime} > 
l_i$ for each $j$ in (\ref{modulop}), and therefore Lemma 3.4(ii) 
applies to the $d_j$ in this union. 
\par Now take $t = p$ in (\ref{apunion2}) and assume for the moment that there 
is some 
pair $(i_1,j_1)$ such that $S(a_{i_1}, p b_{i_1}) = S(c_{j_1}, p b_{j_1})$ and 
$p^{l_{\min}} \mid \mid a_{i_1}$. Then from (\ref{apunion2}) 
it would follow that, for every 
$t \in \Z$, we have the equality of multisets 
\be
\bigcup_{i \neq i_1} S(a_i, t (pb_i)) = \bigcup_{j \neq j_1} S(c_j, t(p b_j)).
\ee
Applying the induction hypothesis on $\mu + \nu$, we could then conclude that 
the arithmetic progressions $S(a_i, p b_i), i \neq i_1$ and 
$S(c_j, pd_j), j \neq j_1$ are equal in pairs. But then, by applying 
Lemma 3.4(ii) 
to any union of the type (\ref{intermediate})-(\ref{modulop}), we would find 
that there must, after all, be a 
pair $(i,j)$ such that $S(a_i,b_i) = S(c_j,d_j)$, so that the induction on 
$\mu + \nu$ yields the proposition.
\par Thus, finally, we may assume there is no pair $(i_1,j_1)$ satisfying the 
above requirements. But, from (\ref{minl}) we know that the progressions 
$S(a_i,p^{l} b_i)$ for which $p^{l} \mid \mid a_i$ and the progreesions
$S(c_j,p^{l}d_j)$ for which $p^{l} \mid \mid c_j$, are equal in pairs. So we 
introduce a third and final 
induction parameter, namely the smallest integer $m$ such that there exists at 
least one pair $(i_1,j_1)$ such that $S(a_{i_1},p^m b_{i_1}) = 
S(c_{j_1}, p^m d_{j_1})$ and $p^l \mid \mid a_{i_1}$. We know that $m$ is finite.
But, if $m > 1$, then applying the previous 
argument for $m=1$ to the multiset relation
\be
\bigcup_{i = 1}^{\mu} S(a_i,t(p^{m-1}b_i)) = 
\bigcup_{j=1}^{\nu} S(c_j,t(p^{m-1} d_j)), \;\;\; {\hbox{for all $t \in \Z$}},
\ee
we could conclude that the progressions 
$S(a_i,p^{m-1} b_i)$ and $S(c_j, p^{m-1} d_j)$ are equal in pairs, thus 
contradicting the definition of $m$. 
\par This final contradiction completes the proof of Proposition 3.3.
\end{proof}

We can now complete the proof of Theorem 1.8. 
Let $i_1,...,i_r$ be the indices for which $q_{1,i} \neq 0$. Our 
goal is to find a pair $u,v$ such that 
$\theta_{i_u} + \theta_{i_v} \in \Z$. Claim 3.2 and Proposition 3.3 
already imply that we can pair off the $\theta_{i_j}$ such that the sum of each 
pair is in $\Z$, modulo their
dependence on $\theta_2,...,\theta_d$. Precisely, let $V_1$ be the 
$\Q$-vector subspace of $V$ spanned by $\theta_2,...,\theta_d$. Then Claim 3.2 
and Proposition 3.3 imply that $r$ is even, say $r=2s$, and the indices 
$i_1,...,i_r$ can be reordered so that, for $t = 1,...,s$, 
\be
q_{1,i_{2t-1}} > 0, \;\;\; q_{1,i_{2t}} = - q_{1,i_{2t-1}}, \;\;\; 
q_{0,i_{2t-1}} + q_{0,i_{2t}} \in \Z
\ee
and hence 
\be
\theta_{i_{2t-1}} + \theta_{i_{2t}} = z_t + v_{1,t}, \;\;\; 
{\hbox{for some $z_t \in \Z$ and $v_{1,t} \in V_1$}}.
\ee
Hence we would be done if we could find any $t$ for which $v_{1,t} = 0$. We can
locate such a $t$ by a more refined application of Lemma 2.1. 
Let $\delta$ be a very small positive real 
number - how small is necessary will again become clear in due course. By 
Lemma 2.1, we can find arbitrarily large integers $n$ satisfying 
\be\label{equi2}
n \equiv 0 \; ({\hbox{mod $L_0$}}) \;\; {\hbox{and}} \;\; \delta^{2i-1} < 
\{\theta_i\} < \delta^{2i-1} + e^{-1/\delta}, \; {\hbox{for $i = 1,...,d$}}.
\ee
Let $M_1$ be the maximum of the numbers $q_{1,i_{2t-1}}$, 
$t = 1,...,s$, and let
$\mathcal{T}_{1} := \{t : q_{1,i_{2t-1}} = M_1\}$. Now let
\be
M_{2,+} := \max \{ q_{2,i_{2t-1}} : t \in \mathcal{T}_1 \}, \;\;\;
M_{2,-} := \min \{q_{2,i_{2t}} : t \in \mathcal{T}_1 \}.
\ee
We claim that $M_{2,-} = - M_{2,+}$. Suppose this is not the case, and without 
loss of generality that $M_{2,+} > - M_{2,-}$. Let $\mathcal{T}_{2} := 
\{t \in \mathcal{T}_1 : q_{2,i_{2t-1}} = M_{2,+}\}$. We shall prove a 
contradiction to the assumption that the function $\epsilon(N)$ is constant 
for $N \gg 0$. Fix a very small $\delta > 0$, let $n_2$ be any integer 
satisfying (\ref{equi2}) and take
\be
N_{2,+} := 2n_2 \cdot \lceil \frac{1}{2(\delta M_1 + \delta^3 M_{2,+})} \rceil, 
\;\;\; N_{2,-} := \frac{N_{2,+}}{2}.
\ee
Then the point is that, provided $\delta$ is small enough, 
\be
(\lfloor N_{2,+} \{\theta_i\} \rfloor, \lfloor N_{2,-} \{\theta_i \} \rfloor) 
= \left\{ \begin{array}{lr} (1,0), & {\hbox{if $i = i_t$ for some $t \in 
\mathcal{T}_{2}$}}, \\
{\hbox{either $(0,0)$ or $(-1,-1)$}}, & {\hbox{otherwise}}.
\end{array} \right.
\ee
Hence,
\be
\epsilon(N_{2,-}) - \epsilon(N_{2,+}) = |\mathcal{T}_{2}| \neq 0,
\ee
giving the desired contradiction, since the numbers $N_{2,\pm}$ can be made 
arbitrarily large. 
\par So we have shown that $M_{2,+} = M_{2,-}$. Let $M_2 := M_{2,+}$. With 
$\mathcal{T}_2$ as defined above we have, for each $t \in \mathcal{T}_2$, that  
\be
q_{\xi,i_{2t-1}} = M_{\xi} = - q_{\xi,i_{2t}}, \; {\hbox{for $\xi = 1,2$}},
\ee
which in turn implies that, if $V_2$ is the $\Q$-vector subspace of $V$ 
spanned by $\theta_3,...,\theta_d$, then, for each $t \in \mathcal{T}_2$, 
\be
\theta_{i_{2t-1}} + \theta_{i_{2t}} = z_t + v_{2,t}, \;\;\; 
{\hbox{for some $z_t \in \Z$ and $v_{2,t} \in V_2$}}.
\ee
The idea now is to iterate the same kind of argument to produce a sequence of 
non-empty sets of indices
\be
\mathcal{T}_1 \supseteq \mathcal{T}_2 \supseteq \cdots \supseteq \mathcal{T}_d
\ee
such that, for any $j = 1,...,d$ and any $t \in \mathcal{T}_j$, 
\be\label{reduction}
\theta_{i_{2t-1}} + \theta_{i_{2t}} = z_t + v_{j,t}, \;\;\; 
{\hbox{for some $z_t \in \Z$ and $v_{j,t} \in V_j$}}.
\ee
Since $V_d = \{0\}$ we will be done at the $d$:th and final step of this 
process. 
\par We have already described in detail the first two steps of the process, 
but for the sake of completeness, let us describe just one further step. Let 
\be
M_{3,+} := \max \{ q_{3,i_{2t-1}} : t \in \mathcal{T}_2 \}, \;\;\;
M_{3,-} := \min \{q_{3,i_{2t}} : t \in \mathcal{T}_2 \}.
\ee
We claim that $M_{3,-} = - M_{3,+}$. Suppose this is not the case, and without 
loss of generality that $M_{3,+} > - M_{3,-}$. Let $\mathcal{T}_{3} := 
\{t \in \mathcal{T}_2 : q_{3,i_{2t-1}} = M_{3,+}\}$. We shall prove a 
contradiction to the assumption that the function $\epsilon(N)$ is constant 
for $N \gg 0$. Fix a very small $\delta > 0$, let $n_3$ be any integer 
satisfying (\ref{equi2}) and take
\be
N_{3,+} := 2n_3 \cdot \lceil \frac{1}{2(\delta M_1 + \delta^3 M_2 + 
\delta^5 M_{3,+})} \rceil, \;\;\; N_{3,-} := \frac{N_{3,+}}{2}.
\ee
Then the point is that, provided $\delta$ is small enough, 
\be
(\lfloor N_{3,+} \{\theta_i\} \rfloor, \lfloor N_{3,-} \{\theta_i \} \rfloor) 
= \left\{ \begin{array}{lr} (1,0), & {\hbox{if $i = i_t$ for some $t \in 
\mathcal{T}_{3}$}}, \\
{\hbox{either $(0,0)$ or $(-1,-1)$}}, & {\hbox{otherwise}}.
\end{array} \right.
\ee
Hence,
\be
\epsilon(N_{3,-}) - \epsilon(N_{3,+}) = |\mathcal{T}_{3}| \neq 0,
\ee
giving the desired contradiction, since the numbers $N_{3,\pm}$ can be made 
arbitrarily large. 
\par So we have shown that $M_{3,+} = M_{3,-}$. Letting $M_3 := M_{3,+}$ and 
with $\mathcal{T}_3$ as above, we have shown that
\be
t \in \mathcal{T}_3 \Rightarrow q_{\xi,i_{2t-1}} = M_{\xi} = - q_{\xi,i_{2t}}, 
\;\; {\hbox{for $\xi = 1,2,3,$}}
\ee
from which (\ref{reduction}) immediately follows for $j = 3$. 
\par Hence, as we have already noted, by iterating the argument as far as 
$j = d$ we will find that, for any $t \in \mathcal{T}_d$, 
$\theta_{i_{2t-1}} + \theta_{i_{2t}} \in \Z$. Since the set 
$\mathcal{T}_d$ will certainly be non-empty, the proof of Theorem 1.8 is 
complete.  
\\
\par We close this section by indicating how to prove Theorem 1.9. In the
notation of the statement of that theorem, if all the $\alpha_i$ are 
irrational, then the result follows immediately from Theorem 1.8. So it 
suffices to show that we cannot have an irreducible $m$-EEC in which there
are both rational and irrational moduli present. To accomplish this, it suffices
to show that the irrational moduli must themselves constitute an 
$m^{\prime}$-EEC for some $m^{\prime}$. Let the representation function
$r(N)$ be as in (\ref{repfn}). As before, the requirement is that
$r(N) = m$ for all $N \gg 0$. Let us separate representations of $N$ 
coming from irrational and rational moduli separately and write
\be
r(N) = r_{{\hbox{irr}}}(N) + r_{{\hbox{rat}}}(N).
\ee
Now the point is that, no matter what the rational moduli are, there must be
some $a \in \N$ such that the function $r_{{\hbox{rat}}}(N)$ is constant on
any congruence class modulo $a$. Hence, the same must be true of 
$r_{{\hbox{irr}}}(N)$, for all $N \gg 0$. But now one may
check that this is enough
to be able to push through the entire proof of Theorem 1.8 and
deduce that the irrational moduli can be paired off so that each pair sums to 
an integer. Theorem 1.9 follows at once. 
  
\setcounter{equation}{0}

\section{The inhomogeneous case}

In the previous section, we employed Weyl equidistribution (Lemma 2.1) to 
reduce the characterisation of homogeneous $m$-EEC's with irrational moduli to 
a purely combinatorial problem about multiset unions of arithmetic 
progressions (Proposition 3.3). The first part of this approach carries over 
to the inhomogeneous setting, but the second part seems to be more difficult 
and we do not resolve it to our satisfaction in this paper. Nevertheless, we 
can at least explain why Question 1.11 has a negative answer and why families 
of inhomogeneous $m$-EEC's may have additional structure.
\par We begin with some terminology :
\\
\\
{\bf Definition 4.1.} A {\em system of parameters} 
$\mathcal{S} = (\mu,\bm{a},\bm{b},\bm{\phi})$ shall consist of a positive 
integer $\mu$ and three $\mu$-tuples
\be\label{system}
\bm{a} = (a_1,...,a_{\mu}),\;\; \bm{b} = (b_{1},...,b_{\mu}),\;\; 
\bm{\phi} = (\phi_{1},...,\phi_{\mu}),
\ee
where the $a_i$ are positive integers and the $b_{i}, \phi_i$ any integers. 
We consider all the tuples as unordered, i.e.: we do not distinguish between 
systems based on the same three tuples but with the entries reordered. 
The number $\mu$ is called the {\em size} of the system. 
We say that the aystem is {\em homogeneous} if $\bm{\phi} = 
\bm{0}$, otherwise {\em inhomogeneous}. 
\\
\\
{\bf Definition 4.2.} Let $\mathcal{S} = (\mu,\bm{a},\bm{b},\bm{\phi})$ and 
$\mathcal{S}^{\prime} = (\nu,\bm{c},\bm{d},\bm{\psi})$ be two systems of 
parameters. We say that these two systems are {\em complementary} if, for 
every $t \in \Z$, we have an equality of multisets 
\be
\bigcup_{i=1}^{\mu} S(a_i,\phi_{i} + t b_{i}) = 
\bigcup_{j=1}^{\nu} S(c_j,\psi_{j} + t d_{j}).
\ee

The study of $m$-EEC's of Beatty sequences can be reduced to that of 
complementary systems of parameters. In the case of homogeneous sequences and
systems, this reduction was established in Claim 3.2. The same arguments
carry over to the inhomogeneous setting. Indeed, let notation be as in 
eqs. (3.1)-(3.8) and assume that all $\gamma_i \in \Q$ - the general 
case can also be reduced to this one. Write $\gamma_i = \frac{g_i}{h_i}$, 
a fraction in lowest terms, and set 
\be\label{h}
H := {\hbox{LCM}}\{h_i : i = 1,...,k\}, \;\;\; \gamma_i =: \frac{G_i}{H}.
\ee
Then the analogoue of (\ref{homsys}) in the inhomogeneous setting is
\be\label{systemdef}
\left\{ \begin{array}{lr} a_i := L_0 V_i H, \;\; b_i = - U_i V_i H, \;\;
\phi_i := - L_0 V_i G_i, & {\hbox{if $q_{1,i} > 0$}}, \\
c_j := - L_0 V_j H, \;\; d_j := - U_j V_j H, \;\; \psi_j := - L_0 V_j G_j, 
& {\hbox{if $q_{1,j} < 0$}}. \end{array} \right.
\ee
Using the same methods as in Section 3, one may show that if the function
$\epsilon(N)$ of (\ref{epsilon}) 
is constant for $N \gg 0$, then for all $t \in \Z$ we must have
equality of multisets
\be\label{inhomred}
\bigcup_{q_{1,i} > 0} S(a_i, \phi_i + t b_i) = 
\bigcup_{q_{1,j} < 0} S(c_j, \psi_j + t d_j).
\ee
In fact, it is not hard to see from our earlier analysis that when 
$d = 1$, i.e.: dim$(V) = 2$, then equality in (\ref{inhomred}) for all 
$t \in \Z$ is also sufficient for constancy of $\epsilon(N)$. 
\par 
At this point, there remains a gap in our understanding, since we do not
know what is the $\lq$right' generalisation of Proposition 3.3 to 
inhmogeneous systems of parameters. However, we shall explain why 
Question 1.11 has a negative answer. We need some more termoinology.
\\
\\
{\bf Definition 4.3.} Let $\mathcal{S} = (\mu,\bm{a},\bm{b},\bm{\phi})$ and 
$\mathcal{S}^{\prime} = (\nu,\bm{c},\bm{d},\bm{\psi})$ be two systems of 
parameters. We say that $\mathcal{S}^{\prime}$ is a {\em subsystem} of 
$\mathcal{S}$ if $\nu \leq \mu$ and there is a $\nu$-element subset 
$\{i_1,...,i_{\nu}\}$ of $\{1,...,\mu\}$ such that
\be
\bm{c} = (a_{i_1},...,a_{i_{\nu}}),\;\; 
\bm{d} = (b_{i_1},...,b_{i_{\nu}}),\;\; 
\bm{\psi} = (\phi_{i_1},...,\phi_{i_{\nu}}).
\ee
A {\em decomposition} of $\mathcal{S}$ is a collection 
$\mathcal{S}^{1},...,\mathcal{S}^{k}$ of subsystems of $\mathcal{S}$ based on 
index sets whose disjoint union is $\{1,...,\mu\}$. We write 
\be\label{decomp}
\mathcal{S} = \bigsqcup_{i=1}^{k} \mathcal{S}^{i}.
\ee
The decomposition is said to be {\em trivial} if $k = 1$, otherwise 
{\em non-trivial}. It is {\em complete} if each $\mathcal{S}^{i}$
has size one.
\\
\\
{\bf Definition 4.4.} A system of parameters $\mathcal{S} = (\mu,\bm{a},\bm{b},
\bm{\phi})$ is said to be {\em exact} if, for each $t \in \Z$, the multiset
$\cup_{i=1}^{\mu} S(a_i,\phi_i + t b_i)$ is an exact cover of the underlying set,
in other words, if every integer occurring in the multiset occurs the same
number of times. 
\par A decomposition (\ref{decomp}) of $\mathcal{S}$ is called {\em exact}
if each $\mathcal{S}^{i}$ is exact. Note that any complete decomposition is
exact, but the converse need not be true.  
\\
\\
{\bf Drfinition 4.5.} A pair  
$(\mathcal{S},\mathcal{S}^{\prime})$ of complementary systems is said to be
{\em reducible/exact/} \\ {\em completely reducible} if there are 
non-trivial/exact/complete decompositions
\be
\mathcal{S} = \bigsqcup_{i=1}^{k} \mathcal{S}^{i}, \;\;\;
\mathcal{S}^{\prime} = \bigsqcup_{i=1}^{k} (\mathcal{S}^{\prime})^{i}
\ee
for which the pairs $(\mathcal{S}^{i},(\mathcal{S}^{\prime})^{i})$, $i = 1,...,k$,
are each complementary/exact/equal.
\\
\par Proposition 3.3 states that any complementary pair of 
homogeneous systems of parameters is completely reducible. 
In general, however, a complementary pair need be neither reducible nor
exact - see Example 4.8 below. 
Together with the following fact, this explains why Question 1.11 has a 
negative answer :
\\
\\
{\bf Proposition 4.6.} {\em If an (irreducible) $m$-EEC with irrational moduli 
has the form 
(\ref{grahamstruct}) then, with notation as in Sections 2-4, the
systems of parameters $\mathcal{S} = (\mu,\bm{a},\bm{b},\bm{\phi})$ and
$\mathcal{S}^{\prime} = (\nu,\bm{c},\bm{d},\bm{\psi})$ defined by 
(\ref{systemdef}) form an exact (irreducible) complementary pair. The latter 
condition is also
sufficient when dim$(V) = d+1 = 2$. In fact, the notations in 
(\ref{grahamstruct}) and (\ref{systemdef}) are consistent, up to a normalising
factor and shifts $(\bm{\phi} \mapsto \bm{\phi} + t \bm{b}, \bm{\psi}
\mapsto \bm{\psi} + t \bm{d})$.}
\\
\par
The verification of these assertions is a tedious recapitulation of earlier
work. We shall therefore content ourselves with giving two further
examples. The first illustrates the correspondences in Proposition 4.6, the
second demonstrates the existence of inexact complementary pairs and hence
of $m$-EEC's not of the form (\ref{grahamstruct}). 
\\
\\
{\bf Example 4.7.} Let $\alpha \in (1,\infty) \backslash \Q$. 
Then $\{S(\alpha,0),
S(\frac{\alpha}{\alpha - 1},0)\}$ is an EEC by Beatty's theorem. Two exact 
covers of $\Z$ by APs are given by 
\be\label{apcovers}
\{S(3,0), S(3,1), S(3,2)\} \;\; {\hbox{and}} \;\; \{S(2,0), S(4,1), S(4,3)\}.
\ee
From this data we can build, as in (\ref{grahamstruct}), the following
irreducible, inhomogeneous \\ EEC :
\be
\left\{ S(3\alpha,0), S(3\alpha,\alpha), S(3\alpha,2\alpha) \right\} \;
\cup \;
\left\{ 
S\left(\frac{2\alpha}{\alpha - 1},\frac{\alpha}{\alpha - 1}\right),
S\left(\frac{4\alpha}{\alpha - 1},\frac{\alpha}{\alpha - 1}\right),
S\left(\frac{4\alpha}{\alpha - 1},\frac{3\alpha}{\alpha - 1}\right) 
\right\}.
\ee

In the notation of (\ref{deftheta}), we have $k = 6$ and the following table
of values  

\begin{table}[ht!]
\begin{center}
\begin{tabular}{|c|c|c|} \hline
$i$ & $\theta_i$ & $\gamma_i$ \\ \hline \hline
$1$ & $\frac{1}{3\alpha}$ & $0$ \\ \hline
$2$ & $\frac{1}{3\alpha}$ & $- 1/3$ \\ \hline
$3$ & $\frac{1}{3\alpha}$ & $- 2/3$ \\ \hline
$4$ & $\frac{\alpha - 1}{2\alpha}$ & $- 1/2$ \\ \hline
$5$ & $\frac{\alpha - 1}{4\alpha}$ & $- 1/4$ \\ \hline
$6$ & $\frac{\alpha - 1}{4\alpha}$ & $- 3/4$ \\ \hline
\end{tabular} 
\end{center}
\end{table}
$\;$ \\
Then (\ref{q}) will become 
\be
\theta_1 = \theta_2 = \theta_3, \;\; \theta_4 = -\frac{3}{2} \theta_1 + 
\frac{1}{2}, \;\; \theta_5 = \theta_6 = \frac{1}{2} \theta_1.
\ee
In (\ref{l}) and (\ref{h}) we'll obtain the values
\be
L_0 = 4, \;\;\; L = 3, \;\;\; H = 12,
\ee
and for the remaining variables in (\ref{uv}), (\ref{h}) and (\ref{systemdef})
the table of values 

\begin{table}[ht!]
\begin{center}
\begin{tabular}{|c|c|c|c|c|c|c|c|c|c|} \hline
$i$ & $U_i$ & $V_i$ & $G_i$ & $a_i$ & $b_i$ & $\phi_i$ & $c_i$ & $d_i$ & 
$\psi_i$ \\ \hline \hline
$1$ & $0$ & $3$  & $0$ & $144$ & $0$  & $0$   & $\;$  & $\;$ & $\;$  \\ \hline
$2$ & $0$ & $3$  & $8$ & $144$ & $0$  & $-96$ & $\;$  & $\;$ & $\;$  \\ \hline
$3$ & $0$ & $3$  & $4$ & $144$ & $0$  & $-48$ & $\;$  & $\;$ & $\;$  \\ \hline
$4$ & $2$ & $-2$ & $0$ & $\;$  & $\;$ & $\;$  & $96$  & $48$ & $0$   \\ \hline
$5$ & $1$ & $-4$ & $9$ & $\;$  & $\;$ & $\;$  & $192$ & $48$ & $144$ \\ \hline
$6$ & $1$ & $-4$ & $3$ & $\;$  & $\;$ & $\;$  & $192$ & $48$ & $48$  \\ \hline
\end{tabular} 
\end{center}
\end{table}

Dividing everything by the normalising 
factor of $48$, we see that (\ref{inhomred})
becomes the assertion that, for every $t \in \Z$, 
\be
S(3,0) \cup S(3,-2) \cup S(3,-1) = S(2,t) \cup S(4,3+t) \cup S(4,1+t).
\ee
Notice that this equality is irreducible and that, when $t = 0$, it 
coincides with 
that between the
pair of exact covers we started with in (\ref{apcovers}).  
\\
\\
{\bf Example 4.8.} Let $\mathcal{S} = (\mu,\bm{a},\bm{b},\bm{\phi})$ and
$\mathcal{S}^{\prime} = (\nu,\bm{c},\bm{d},\bm{\psi})$ be systems for which
\be\label{t-remove}
b_i \equiv 0 \; ({\hbox{mod $a_i$}}), \; i = 1,...,\mu, \;\;\;
c_j \equiv 0 \; ({\hbox{mod $d_j$}}), \; j = 1,...,\nu.
\ee
Then both sides of (\ref{inhomred}) are independent of $t$, so it
suffices for complementarity to have the multiset equality
\be\label{t-indep}
\bigcup_{i=1}^{\mu} S(a_i, \phi_i) = \bigcup_{j=1}^{\nu} S(c_j, \psi_j).
\ee

Consider the solution of (\ref{t-indep}) given by 
\be\label{exa}
S(1,0) \cup S(6,0) = S(2,0) \cup S(3,0) \cup S(6,1) \cup S(6,5).
\ee
One readily checks that this equality is irreducible and inexact. Hence
any corresponding complementary pair of systems 
satisfying (\ref{t-remove}) will be
both irreducible and inexact. This is the simplest example we found of an 
inexact complementary pair, in that the value of $\mu + \nu = 6$ is minimal
(note that one must have $\min \{\mu,\nu\} > 1$), and 
likewise with the moduli $a_i,c_j$.
\par We can use this data to construct an irreducible 
$2$-EEC of Beatty sequences
with irrational moduli, which does not have the form (\ref{grahamstruct}). 
In the notation of (\ref{q}), we choose $d = 1$, $k = 6$. Condition
(\ref{t-remove}) will be satisfied if $q_{0,i} \in \Z$ for all $i$. Then it
is easy to check that, with the following assignments, (\ref{systemdef})
reduces (\ref{t-indep}) to (\ref{exa}) :

\newpage

\begin{table}[ht!]
\begin{center}
\begin{tabular}{|c|c|c|} \hline
$i$ & $\theta_i$        & $\gamma_i$ \\ \hline \hline
$1$ & $\theta_1$        & $0$        \\ \hline
$2$ & $z_2 + 6\theta_1$ & $0$        \\ \hline
$3$ & $z_3 - 2\theta_1$ & $0$        \\ \hline
$4$ & $z_4 - 3\theta_1$ & $0$        \\ \hline
$5$ & $z_5 - \theta_1$  & $1/6$      \\ \hline
$6$ & $z_6 - \theta_1$  & $5/6$      \\ \hline
\end{tabular}
\end{center}
\end{table}    
$\;$ \par
Here $\theta_1$ is any positive irrational and the $z_i$ are integers. By
(\ref{sumq}), we have $m = z_2 + \cdots + z_6$. Since each $\theta_i > 0$, 
the minimum possible value of $m$ is thus $m = 4$, obtained by choosing
$z_2 = 0$, $z_3 = z_4 = z_5 = z_6 = 0$ and $\theta_1 < 1/3$. This will
yield an irreducible 
$4$-EEC of Beatty sequences with irrational moduli, which does not 
have the form (\ref{grahamstruct}). However, as promised above, we can 
do better and construct an irreducible 
$2$-EEC instead. The point is that, formally, in 
the proof of Claim 3.2, 
there is no requirement that the $\theta_i$ in (\ref{q}) be positive,
and also nothing changes if we shift any $\theta_i$ by an integer. So, if
we set $\theta := -\theta_1$, we can define a new family of Beatty sequences by 

\begin{table}[ht!]
\begin{center}
\begin{tabular}{|c|c|c|} \hline
$i$ & $\theta^{\prime}_{i}$ & $\gamma_i$ \\ \hline \hline
$1$ & $1 + \theta$        & $0$        \\ \hline
$2$ & $1 + 6\theta$       & $0$        \\ \hline
$3$ & $-2\theta$          & $0$        \\ \hline
$4$ & $-3\theta$          & $0$        \\ \hline
$5$ & $-\theta$           & $1/6$      \\ \hline
$6$ & $-\theta$           & $5/6$      \\ \hline
\end{tabular}
\end{center}
\end{table}   
$\;$ \par
This yields an irreducible 
$2$-EEC provided $-1/6 < \theta < 0$. By the way, consider the
function $\epsilon(N)$ of (\ref{epsilon}). Let $x := \{N\theta\}$. Then 
\be
\epsilon(N) = f(x) = \{x\} + \{6x\} + \{-2x\} + \{-3x\} + \{-x + 1/6\} + 
\{-x + 5/6\}.
\ee
Since $\{N\theta\}$ is equidistributed in $[0,1)$, constancy of $\epsilon(N)$
for $N \gg 0$ is equivalent to constancy of $f(x)$ for $x \in [0,1)$. One 
readily checks that $f(x) = 2$ for all $x \in [0,1)$. 
\\
\par
In general, given an irreducible and inexact solution to (\ref{t-indep}), one 
can construct, as in 
Example 4.8, a corresponding irreducible $m$-EEC not of the
form (\ref{grahamstruct}), where $m = \min \{\mu,\nu\}$. It is
easy to see how (\ref{exa}) can be generalised to give examples of
irreducible and inexact solutions of (\ref{t-indep}), for any 
value of $\min\{\mu,\nu\} > 1$. Hence we deduce 
\\
\\
{\bf Theorem 4.9.} {\em For every $m > 1$, there exist irreducible
$m$-EEC's of Beatty sequences with irrational moduli, not having the form
(\ref{grahamstruct})}.

\setcounter{equation}{0}

\section{A fractional Beatty theorem}
 
The notion of exact 
$m$-cover in Definition 1.2 clearly does not make sense if $m$ is
not an integer. However, one might imagine various ways of extending the
notion to non-integer $m$. Here, we only take a first tentative step, which
nevertheless may prove instructive. We shall prove a $\lq$fractional version'
of Beatty's theorem.
\par Let $p,q$ be relatively prime positive integers. Let $\alpha_1, \alpha_2$ 
be positive irrationals satisfying 
\be\label{beatty}
\frac{1}{\alpha_1} + \frac{1}{\alpha_2} = \frac{p}{q}.
\ee
As in Section 2, denote $\theta_i := 1/\alpha_i, \; i = 1,2$. 
Let $p_0,p_1 \in \{0,1,...,q-1\}$ be the integers defined by 
\be
p \equiv p_0 \; ({\hbox{mod $q$}}), \;\;\;\; \frac{p_1}{q} < \{\theta_1\} < 
\frac{p_1 + 1}{q}.
\ee
Let $r(N)$ be the representation function of (\ref{repfn}) and set
\be
R(N) := \sum_{M=1}^{N} r(M).
\ee
We will prove the following result :
\\
\\
{\bf Theorem 5.1.}
{\em For every $N \in \mathbb{N}$ one has 
\be\label{R}
R(qN-1) = \left\{ \begin{array}{lr} pN - \lceil p/q \rceil, & {\hbox{if 
$p_1 < p_0$}}, \\ pN - \lfloor p/q \rfloor, & {\hbox{if $p_1 \geq p_0$}}.
\end{array} \right.
\ee
Moreover, 
\\
\\
(A) If $q = 1$, then $r(N) = p$ for every $N \in \mathbb{N}$. 
\\
(B) If $q = 2$, then $r(N) \in \{ \lfloor p/2 \rfloor, \lceil p/2 \rceil \}$
for every $N \in \mathbb{N}$. 
\\
(C.i) If $q > 2$ and $p_1 < p_0$, then 
\be\label{three}
r(N) \in \left\{ \lfloor p/q \rfloor, \lceil p/q \rceil, \lceil p/q \rceil + 1
\right\}, \;\;\; {\hbox{for every $N \in \mathbb{N}$}}.
\ee
If, for each $i = 0,1,2$, we let 
\be
S_i := \{ N \in \mathbb{N} : r(N) = \lfloor p/q \rfloor + i \},
\ee
then each $S_i$ has asymptotic density, say $d(S_i) = d_i$, where
\begin{eqnarray}
d_0 = \left( 1 - \frac{p_0}{q} \right) + d_2, \\
d_1 = \frac{p_0}{q} - 2d_2, \\
d_2 = \frac{1}{q} \left[ \frac{p_{0}^{2} + p_{1}^{2} - (p_0 - p_1)}{2q} - 
p_1 \{\theta_1\}\right].
\end{eqnarray}
(C.ii) If $q > 2$ and $p_1 \geq p_0$, then 
\be
r(N) \in \left\{ \lfloor p/q \rfloor - 1, \lfloor p/q \rfloor, 
\lceil p/q \rceil \right\}, \;\;\; {\hbox{for every $N \in \mathbb{N}$}}.
\ee
If, for each $i = 0,1,2$, we let 
\be
T_i := \{ N \in \mathbb{N} : r(N) = \lceil p/q \rceil -i \},
\ee
then each $T_i$ has asymptotic density, say $d(T_i) = \delta_i$, where
\begin{eqnarray}
\delta_0 = \frac{p_0}{q} + \delta_2, \\
\delta_1 = \left( 1 - \frac{p_0}{q} \right) - 2\delta_2, \\
\delta_2 = \frac{1}{q} \left[ \frac{4p_0 p_1 + (p_1 - p_0) - (p_{0}^{2} + 
p_{1}^{2})}{2q} - (2p_1 - p_0) + (q-p_1)\{\theta_1\} \right].
\end{eqnarray}}
  
{\bf Remark 5.2.} The interesting thing in this result is that, when
$q > 2$, the function $r(N)$ cannot take on just the values $\lfloor p/q 
\rfloor$ and $\lceil p/q \rceil$. Nevertheless, $r(N)$ 
never takes on more than three distinct values, and each value is assumed
on a fairly regular set. Thus the family
$\{S(\alpha_1,0), S(\alpha_2,0)\}$ is always, in some sense, $\lq$close to an
exact $p/q$-cover'.  

\begin{proof}
Eqs. (\ref{epsilon}) and (\ref{rep}) here become 
\begin{eqnarray}
\epsilon(N) = \{N\theta_1\} + \{N\theta_2\}, \\
r(N) = \frac{p}{q} + (\epsilon(N) - \epsilon(N+1)). 
\end{eqnarray}
Define $c_N \in \{0,1,...,q-1\}$ by $c_N \equiv Np \; ({\hbox{mod $q$}})$. 
Since (\ref{beatty}) implies that
\be
N\theta_1 + N\theta_2 \equiv \frac{c_N}{q} \; ({\hbox{mod $1$}}),
\ee
it follows that 
\be\label{epsn}
\epsilon (N) = \left\{ \begin{array}{lr} c_N / q, & {\hbox{if $\{N\theta_1\} < 
c_N / q$}}, \\ 1 + c_N / q, & {\hbox{if $\{N\theta_1\} > c_N / q$}}. 
\end{array} \right.
\ee
In particular, 
\be\label{epsone}
\epsilon(1) = \left\{ \begin{array}{lr} p_0 / q, & {\hbox{if $p_1 < p_0$}}, \\
1 + p_0 / q, & {\hbox{if $p_1 \geq p_0$}}, \end{array} \right.
\ee
whereas
\be\label{qdivn}
{\hbox{if $q | N$, then $c_N = 0$ and $\epsilon (N) = 1$}}.
\ee 
From (5.16) it follows that, for any $N_1 > N_2$,
\be
R(N_1) - R(N_2) = \left( \frac{p}{q} \right)(N_1 - N_2) - (\epsilon(N_1 + 1)
- \epsilon (N_2 + 1)).
\ee
In particular, if $N \equiv -1 \; ({\hbox{mod $q$}})$, then (\ref{qdivn}) 
implies that
\be\label{Rdiff}
R(N+q) - R(N) = pN.
\ee
Furthermore,  
\be\label{Rqminusone}
R(q-1) = \frac{p(q-1)}{q} + (\epsilon (1) - 1) = p - \lfloor p/q \rfloor - 
p_0 / q + (\epsilon (1) - 1).
\ee
From (\ref{epsone}), (\ref{Rdiff}) and (\ref{Rqminusone}), 
one easily deduces (\ref{R}). 
Now we turn to the proofs of
statements (A), (B) and (C). The first of these 
is just Theorem 1.7, and it
is immediately implied by (5.16) and (\ref{qdivn}). 
Using (\ref{epsn}) we also quickly deduce (B).
For (C) we need to work a little more. We shall prove the 
statements of (C.i) rigorously - similar arguments give (C.ii). 
It is already clear from (5.16)
that $r(N)$ must be one of the four numbers $\lfloor p/q \rfloor + i$, $i \in
\{-1,0,1,2\}$. If $r(N) = \lfloor p/q \rfloor - 1$ 
then it means that $\epsilon(N+1) - \epsilon(N) = 1 + \frac{p_0}{q}$. By
(\ref{epsn}), this happens if and only if 
\be
\{N\theta_1\} < \frac{c_N}{q}, \;\;\; \{(N+1)\theta_1\} > \frac{c_{N+1}}{q}, 
\;\;\; c_{N+1} = c_N + p_0 < q.
\ee
In particular, these conditions are unsatisfiable if $\{\theta_1\} < 
\frac{p_0}{q}$, in other words if $p_1 < p_0$. 
This proves (\ref{three}). The set $S_2$ consists of
all those $N \in \mathbb{N}$ for which $\epsilon(N+1) - \epsilon(N) = 
\frac{p_0}{q} - 2$. By (\ref{epsn}), we have explicitly,
\be\label{stwo}
S_2 = \left\{ N \in \mathbb{N} : \{N\theta_1\} > \frac{c_N}{q}, \;\; 
\{(N+1)\theta_1\} < \frac{c_{N+1}}{q}, \;\; 
c_{N+1} = c_N + p_0 - q \geq 0 \right\}.
\ee
That this set has
an asymptotic density follows from Lemma 2.1, which we can also use to compute 
$d_2$ explicitly. Note that (5.7) and (5.8) would follow from (5.9) and 
the fact that
\be
d_0 + d_1 + d_2 = 1.
\ee
Hence, it just remains to compute $d_2$. From (\ref{stwo}) we deduce that
$N \in S_2$ if and only if 
\be
q-p_0 \leq c_N \leq q-1
\ee
and
\be
\max \left\{ \frac{c_N}{q}, 1 - \{\theta_1\} \right\} < \{N\theta_1\} < 
\frac{c_N + p_0}{q} - \{\theta_1\}.
\ee
By Lemma 2.1, we thus have 
\be
d_2 = \frac{1}{q} \left[ \sum_{j = q-p_0}^{q-p_1 - 1} \left( 
\frac{j+p_0-q}{q} \right) + \sum_{j=q-p_1}^{q-1} \left( \frac{p_0}{q} - 
\{\theta_1\} \right) \right].
\ee
It is now just a tedious exercise to verify (5.9). 
\end{proof}   
      
\setcounter{equation}{0}

\section{Open questions}

In this paper we showed how the classification of $m$-EEC's of Beatty 
sequences with irrational moduli can be reduced to that of
complementary pairs of systems of parameters, the latter problem being 
purely arithmetical. We proved that every homogeneous complementary pair is
completely reducible, but that there exist inhomogeneous complementary pairs
which are neither reducible nor exact. There one might let things rest, 
but we feel that something is still missing, that it should be
possible to prove some more insightful structural result for arbitrary
complementary pairs. This is admittedly a vague hypothesis. Equally vague, 
but still enticing, is the question of how 
to push further the notion of $m$-cover, when $m$ is
not an integer. Theorem 5.1 may provide some hints, but let us stop 
before we cross over the threshold into the
realm of idle speculation !
 
\section*{Acknowledgements}

I thank Aviezri Fraenkel and Urban Larsson for stimulating discussions. 
My research is partly supported by a grant from the Swedish Science
Research Council (Vetenskapsr\aa det).

\end{document}